\documentclass[letterpaper, 10 pt, conference]{ieeeconf}  

\IEEEoverridecommandlockouts                              
\usepackage{graphics} 
\usepackage{epsfig} 
\usepackage{times} 
\usepackage[cmex10]{amsmath} 
\usepackage{amssymb}  
\usepackage{nccmath}
\usepackage{color}
\usepackage{bm}

\newtheorem{proposition}{Proposition}
\newtheorem{assumption}{Assumption}

\newcommand{\dd}{\mathop{}\!\mathrm{d}}
\long\def\/*#1*/{}

\title{\LARGE \bf Transient Performance of Electric Power Networks under Colored Noise}

\author{T. Coletta, B. Bamieh, and Ph. Jacquod 
\thanks{T. Coletta and Ph. Jacquod are with the School of Engineering of the University of Applied Sciences of Western Switzerland
CH-1951 Sion, Switzerland. Emails: \tt\small{(tommaso.coletta, philippe.jacquod)@hevs.ch}}
\thanks{B. Bamieh is with the Department of Mechanical Engineering at the University of California at Santa Barbara, Santa Barbara, CA 93106, USA
Email: \tt\small{bamieh@engr.ucsb.edu}}}

\begin{document}

\maketitle
\thispagestyle{empty}
\pagestyle{empty}

\begin{abstract}
New classes of performance measures have been recently introduced to quantify the transient response to external disturbances
of coupled dynamical systems on complex networks. 
These performance measures are time-integrated quadratic forms in the system's coordinates or their time derivative.
So far, investigations of these performance measures have been restricted to Dirac-$\delta$ impulse disturbances, in which case
they can be alternatively interpreted as giving the long time output variances for stochastic white noise power 
demand/generation fluctuations.
Strictly speaking, the approach is therefore restricted to power 
fluctuating on time scales shorter than the shortest time scales in the swing equations. To account for 
power productions from new renewable energy sources, 
we extend these earlier works to the relevant case of colored noise power fluctuations, with a finite correlation time $\tau > 0$. 
We calculate a closed-form expression for generic quadratic performance measures. Applied to specific cases, this leads
to a spectral representation of performance measures as a sum over the non-zero modes of the network Laplacian.
Our results emphasize the competition between inertia, damping and the Laplacian modes, whose balance is determined to a large
extent by the noise correlation time scale $\tau$.
\end{abstract}


\section{Introduction}
Transmission system operators constantly monitor electric power grids and evaluate their potential response to 
possible faults and unexpected disturbances~\cite{Bialek08}. Standardly, frequency nadir and maximum rate of change of frequency (RoCoF)
are the indicators of choice. This is so because when they exceed pre-defined values, they trigger corrective measures such as activation of primary control
and disconnection of power plants from the grid, or even,  in worst cases, of entire geographical areas. Frequency nadir and RoCoF however 
only partially characterize the transient excursion away from the previously operating synchronous state. Inspired by consensus and 
synchronization studies~\cite{Bamieh12,Summers15,Siami14,Fardad14}, recent works have introduced new transient performance measures 
following disturbances on busses~\cite{bamieh2013price, Gayme15, Gayme16,Poola17,paganini2017global} and on power lines~\cite{coletta2017performance}.
The method is
particularly appealling because (i) these performance measures have a clear physical meaning, quantifying the additional ohmic 
losses~\cite{bamieh2013price, Gayme15, Gayme16}, the primary control effort~\cite{Poola17,coletta2017performance} originating from the transient
or the phase coherence of the grid~\cite{Bamieh12,Summers15,coletta2017performance}; (ii) they are integrated quadratic forms 
in phase or frequency excursions that can be expressed as  
$\mathcal{L}_2$ norms of the system's output. As such, they are conveniently calculated 
through an observability Gramian solution to a Lyapunov equation~\cite{Zhou96}.

Technically speaking, the advantage of the $\mathcal{L}_2$-based approach 
over infinity norm measures (frequency nadir and RoCoF) is not only that it is mathematically tractable,
but also that it provides combined information on both the amplitude and the duration of the transient in both voltage phases and frequencies.
Its drawback is that so far it has been applied to Dirac-$\delta$ impulse disturbances only~\cite{bamieh2013price, Gayme15, 
Gayme16,Poola17,paganini2017global,coletta2017performance}, which is equivalent to
considering long time output variances for stochastic white noise power demand/generation fluctuations. 
Fluctuations of photovoltaic or wind turbine power generation have however finite correlation times~\cite{schmietendorf2017impact}, typically on the order of minutes
or more, i.e. significantly longer than typical time scales in swing equations. 
To model them it is therefore desirable to go beyond white noise power fluctuations.
In this work we extend the observability Gramian formalism to treat the swing dynamics under colored noise inputs.
In the spirit of the method proposed in \cite{zare2017colour} in the different context of turbulent flows, we achieve this by means of a filter to generate 
colored noise from a white noise stochastic input~\cite{Fox88}. We obtain fluctuating power generations with exponentially decaying
correlations, with a tunable characteristic correlation time $\tau$ which we take as a parameter in our model. We
provide a closed form expression for any output representing a quadratic performance index. To
illustrate our theory, we consider a performance measure quantifying voltage angle coherence in the network.
For fluctuating injections localized on a single node, we show how different correlation time scales lead to qualitatively different behaviors of this
quadratic performance measure.

This paper is organized as follows. Section \ref{sec:Notations} introduces some mathematical notations.
Section \ref{sec:Model} presents the model and the observability Gramian formalism for colored noise.
In Sec.~\ref{sec:Expression generic performance} we diagonalize the dynamics and present a closed form expression for generic performance measures.
In Sec.~\ref{sec:Applications} we illustrate our theory for a specific case of an angle coherence performance measure and
discuss the obtained results.
A  brief conclusion is given in Sec. \ref{sec:Conclusion}.

\section{Mathematical notation}\label{sec:Notations}
Given a vector ${\bm v}\in\mathbb{R}^N$ and a matrix ${\bm M}\in \mathbb{R}^{N\times N}$
we denote their transpose by ${\bm v}^\top$ and ${\bm M}^\top$.
In terms of the components $v_1,\ldots,v_N$, we also denote the vector $\bm v$ as ${\bm v}\equiv\textrm{vec}(\{v_i\})$,
while $\textrm{diag}(\{v_i\})\in\mathbb{R}^{N\times N}$ denotes the diagonal matrix having $v_1,\ldots,v_N$ as diagonal entries.
The $l^\textrm{th}$ unit vector $\hat{{\bm e}}_l\in \mathbb{R}^{N}$ has components $(\hat{e}_l)_i = \delta_{il}$.

We denote undirected weighted graphs by $\mathcal{G}=(\mathcal{N},\mathcal{E}, \mathcal{W})$ where 
$\mathcal{N}$ is the set of $N$ vertices,
$\mathcal{E}$ is the set of edges, and $\mathcal{W}=\{b_{ij}\}$ is the set of edge weights,
with $b_{ij}=0$ whenever $i$ and $j$ are not connected by an edge,
and $b_{ij}=b_{ji}>0$ otherwise.
The graph Laplacian ${\bm L}\in\mathbb{R}^{N\times N}$ is the symmetric matrix with components $L_{ij}=-b_{ij}$ if $i\neq j$
and $L_{ii}=\sum_{i\neq j}b_{ij}$.
We denote by $\{\lambda_1,\ldots\lambda_N\}$ and $\{{\bm u}^{(1)},\ldots,{\bm u}^{(N)}\}$ the eigenvalues and orthonormalized eigenvectors of $\bm L$.
The orthogonal matrix $\bm T\in \mathbb{R}^{N\times N}$ having ${\bm u}^{(i)}$ as $i^\textrm{th}$ column diagonalizes $\bm L$,
i.e. ${\bm T}^\top {\bm L} {\bm T} = {\bm \Lambda}$ where ${\bm \Lambda}=\textrm{diag}(\{\lambda_i\})$.
The zero row and column sum property of $\bm L$ implies that $\lambda_1=0$ and that ${{\bm u}^{(1)}}^\top=[1,\ldots,1]/\sqrt{N}$.
In connected graphs, all remaining eigenvalues of $\bm L$ are strictly positive, $\lambda_i>0$ for $i\geq2$.
The column row vector product $\bm{u}^{(1)}{\bm{u}^{(1)}}^\top$ is the $N\times N$ matrix having $1/N$ for all its entries.

\section{POWER NETWORK MODEL}\label{sec:Model}
We consider the swing dynamics of high voltage transmission power networks in the DC approximation.
This approximation of the full nonlinear dynamics assumes uniform and constant voltage magnitudes, 
purely susceptive transmission lines and small voltage phase differences. 
We consider a Kron reduced network such that each of its $N$ nodes models a synchronous machine (generator or consumer) of 
rotational inertia $m_i>0$ and damping coefficient $d_i>0$.
The steady state power flow equations relating the active power injections $\bm P$ 
to the voltage phases $\bm \theta$ at every node define the nominal operating point ${\bm\theta}^\star$
through $\bm P ={\bm L} {\bm \theta}^\star$. 
Here, ${\bm L}$ is the Laplacian matrix of the graph modeling
the Kron reduced electric network and whose edge weights are given by the effective susceptances $b_{ij}\geq0$.

Subject to a power injection disturbance ${\bm p}(t)$, the system deviates from the nominal operating point 
according to ${\bm \theta}(t)={\bm\theta}^\star+{\bm \varphi}(t)$, and
${\bm \omega}(t)=\dot{\bm \varphi}(t)$.
In the DC approximation, and in a frame rotating at the nominal frequency of the network, the swing equations
read \cite{Bialek08}
\begin{equation}\label{eq:Swing}
{\bm M}\ddot{{\bm \varphi}} = -{\bm D}\dot{{\bm \varphi}} - {{\bm L}}{\bm \varphi} + {\bm p}(t)\,,
\end{equation}
with ${\bm M}=\textrm{diag}(\{m_i\})$ and $\bm D = \textrm{diag}(\{d_i\})$.
Given performance outputs of the form 
\begin{equation}\label{eq:output}
 \bm y= 
 \setlength\arraycolsep{2pt}
 \begin{bmatrix}
	 {{\bm Q}^{(1,1)}} & 0 \\
        0 & {{\bm Q}^{(2,2)}}
	\end{bmatrix}^{1/2}
      \bigg[\begin{array}{cc}
        \bm \varphi \\
        \bm \omega 
      \end{array}\bigg]\,,
\end{equation}
we want to assess the long time output variance
\begin{equation}\label{eq:Long time output variance}
 \lim_{t\rightarrow \infty}\mathbb{E}[\bm y(t)^\top\bm{y}(t)]
\end{equation}
for the swing dynamics (\ref{eq:Swing}) subject to fluctuating 
power injection ${\bm p}(t)$ defined by the correlator
\begin{equation}\label{eq:exp decorrelation}
 \mathbb{E}[p_i(t_1)p_j(t_2)]=\delta_{ij} p_i^2 e^{-|t_1-t_2|/\tau}\,.
\end{equation}
In (\ref{eq:exp decorrelation}), $p_i^2$ denotes the equal time variance of the power injection disturbance at node $i$,
$\tau$ is the characteristic correlation time scale, and $\mathbb{E}$ denotes the expectation value.

Obtaining (\ref{eq:Long time output variance}) is equivalent to considering Dirac$-\delta$ impulse disturbances 
in the augmented dynamical system (see Appendix \ref{app:Colored noise})
\begin{equation}\label{eq:Swing2}
      \begin{bmatrix}
        \dot{\bm \varphi} \\
        \dot{\bm \omega} \\ 
        \dot \eta
      \end{bmatrix}=
      \setlength\arraycolsep{3pt}
       \begin{bmatrix}
        0 & \mathbb{I} & 0 \\
        -{\bm M}^{\text{-}1}{\bm L} & -{\bm M}^{\text{-}1}{\bm D} & {\bm M}^{\text{-}1}\bm p \\
        0 & 0 & -\tau^{\text{-1}}
      \end{bmatrix}
      \begin{bmatrix}
        \bm \varphi \\
        \bm \omega  \\
        \eta
      \end{bmatrix}
      +\begin{bmatrix}
        0 \\
        0 \\
        \delta(t)\eta_0
      \end{bmatrix}\,,
\end{equation}
with $\eta_0=\sqrt{2/\tau}$, and measuring the transient performance by evaluating the quadratic measure
\begin{equation}\label{eq:objective function}
\mathcal{P}= \int_{0}^{\infty}
 \begin{bmatrix}
        \bm \varphi^\top
        \bm \omega^\top 
        \eta
      \end{bmatrix}{\bm Q}
 \begin{bmatrix}
        \bm \varphi \\
        \bm \omega \\
        \eta
      \end{bmatrix}\dd t\,, \quad
      \bm Q=
\setlength\arraycolsep{2pt}
      \begin{bmatrix}
        {\bm Q}^{(1,1)} & 0 & 0 \\
        0 & {\bm Q}^{(2,2)} & 0 \\
	0 & 0 & 0 \\
        \end{bmatrix} . \quad
\end{equation}
which remains finite under the assumption that ${\bf u}^{(1)} \in {\rm ker}\,  \bm Q^{(1,1)}$.
Following~\cite{paganini2017global} and~\cite{coletta2017performance}, we introduce the change of variables
$\overline{{\bm \varphi}}={\bm M}^{1/2}{\bm \varphi}$ and
$\overline{{\bm \omega}}={\bm M}^{1/2}{\bm \omega}$. This allows us to rewrite (\ref{eq:Swing2}) as 
\begin{equation}\label{eq:Swing3}
      \begin{bmatrix}
        \dot{\overline{\bm \varphi}} \\
        \dot{\overline{\bm \omega}} \\ 
        \dot \eta
      \end{bmatrix}=
      \setlength\arraycolsep{3pt}
       \underbrace{\begin{bmatrix}
        0 & \mathbb{I} & 0 \\
        -{\bm L}_\textrm{M} & -{\bm M}^{\text{-}1}{\bm D} & {\bm M}^{\text{-}1/2}\bm p \\
        0 & 0 & -\tau^{\text{-1}}
      \end{bmatrix}}_{\bm A}
      \begin{bmatrix}
        \overline{\bm \varphi} \\
        \overline{\bm \omega}  \\
        \eta
      \end{bmatrix}
      +\begin{bmatrix}
        0 \\
        0 \\
        \delta(t)\eta_0
      \end{bmatrix}\,,
\end{equation}
where ${\bm L}_\textrm{M}$ is the symmetric matrix
\begin{equation}\label{eq:Def of Lm}
 {\bm L}_\textrm{M} = {\bm M}^{\text{-}1/2}{\bm L}{\bm M}^{\text{-}1/2}\,.
\end{equation}
For initial conditions $(\bm \varphi(0),\bm \omega(0), \eta(0))=(0,0,0)$, solving (\ref{eq:Swing3}) yields
\begin{equation}\label{eq: Definition B}
  \begin{bmatrix}
        \overline{\bm \varphi}(t) \\
        \overline{\bm \omega}(t) \\
        \eta
  \end{bmatrix} = e^{{\bm A} t}\underbrace{
    \begin{bmatrix}
        0 \\
        0 \\
        \eta_0
    \end{bmatrix}}_{\bm B}.
\end{equation}
The performance measure (\ref{eq:objective function}) can be expressed as 
\begin{equation}\label{eq:objective function 2}
\mathcal{P}={\bm B}^\top \bm X {\bm B}\,,
\end{equation}
with the observability Gramian $ {\bm X} = \int_{0}^{\infty}e^{{\bm A}^\top t} \bm Q_\textrm{M} e^{{\bm A}t} \dd t$,
and
\begin{equation}\label{eq:generic perf measure}
 {\bm Q}_\textrm{M} = 
      \setlength\arraycolsep{2pt}
      \begin{bmatrix}
        {\bm M}^{\text{-}1/2}{\bm Q}^{(1,1)}{\bm M}^{\text{-}1/2} & 0 & 0 \\
        0 & {\bm M}^{\text{-}1/2}{\bm Q}^{(2,2)}{\bm M}^{\text{-}1/2} & 0 \\ 
        0 & 0 & 0
      \end{bmatrix}\,.
\end{equation}
In what follows we denote the non zero blocks of ${\bm Q}_\textrm{M}$ by 
${\bm Q}^{(1,1)}_\textrm{M}={\bm M}^{\text{-}1/2}{\bm Q}^{(1,1)}{\bm M}^{\text{-}1/2}$,
and ${\bm Q}^{(2,2)}_\textrm{M}={\bm M}^{\text{-}1/2}{\bm Q}^{(2,2)}{\bm M}^{\text{-}1/2}$ respectively.

When the matrix ${\bm A}$ is Hurwitz, the system is asymptotically stable and the observability Gramian ${\bm X}$ satisfies
the Lyapunov equation
\begin{equation}\label{eq:Lyapunov equation}
{\bm A}^\top {\bm X} + {\bm X} {\bm A} = -{\bm Q}_\textrm{M}\,.
\end{equation}
In the present case however, the system is Laplacian and it follows from ${\bm L} {\bm u}^{(1)} = 0$
that ${\bm A}$ has a marginally stable mode ${\bm A}[{\bm M}^{1/2}{{\bm u}^{(1)}}, 0, 0]^\top=0$.
Nevertheless, this unobservable mode does not carry any relevant physical information and only reflects
the model's invariance under a global shift of all voltage phases.
Standard approaches to deal with this marginally stable mode include: (i) considering performance measures $\bm Q$
such that ${\bm u}^{(1)}\in\textrm{ker}(\bm Q^{(1,1)})$, in which case the observability Gramian is 
well defined by~(\ref{eq:Lyapunov equation}) with the additional constraint ${\bm X}[{\bm M}^{1/2}{{\bm u}^{(1)}}, 0, 0]^\top=0$
\cite{bamieh2013price, Gayme16, Poola17}, and (ii) introducing a regularizing parameter, $\epsilon$, in the Laplacian making it nonsingular,
and taking the limit $\epsilon\rightarrow0$ only at the very end of the calculation of a performance measure \cite{coletta2017performance}. 
In the derivations of this manuscript we will follow the latter approach. 

\section{CLOSED FORM EXPRESSION FOR QUADRATIC PERFORMANCE MEASURES}\label{sec:Expression generic performance}
Before providing a closed form expression for performance measures of the type (\ref{eq:objective function}),
we first recall two results proven in \cite{coletta2017performance}.\newline
\begin{proposition}[Laplacian regularization]\label{prop:Marginal mode assumption}\ \\
 Under the transformation ${\bm L}\rightarrow {\bm L}+\epsilon\mathbb{I}$, with regularizing parameter $\epsilon>0$, the system
 defined in~(\ref{eq:Swing3}) is asymptotically stable and has no marginally stable mode. 
\end{proposition}
\begin{proposition}[Solution of the Lyapunov equation]\label{prop:Sol Lyapunov eq}\ \\
 Let $\bm A$ be a non symmetric, diagonalizable matrix with eigenvalues $\mu_i\neq0$. Let 
 ${\bm T}_R$ (${\bm T}_L$) denote the matrix whose columns (rows) are the right (left) eigenvectors of $\bm A$.
 The observability Gramian $\bm X$, solution of the Lyapunov equation~(\ref{eq:Lyapunov equation}) is given by 
 \begin{equation}\label{eq:Solution of lyapunov equation}
 X_{ij} = \sum_{l,q=1}^{2N+1}\frac{-1}{\mu_l+\mu_q}(T_L)_{li}(T_L)_{qj}\left( {\bm T}_R^\top \bm Q^\textrm{M} {\bm T}_R\right)_{lq}\,.
\end{equation}
The proofs of Propositions \ref{prop:Marginal mode assumption} and \ref{prop:Sol Lyapunov eq} are given in \cite{coletta2017performance}.\\
\end{proposition}

Under the transformation of Proposition \ref{prop:Marginal mode assumption}, $\bm A$ has no marginal modes and
the Lyapunov equation~(\ref{eq:Lyapunov equation}) suffices to define the observability Gramian. 
For the regularized Laplacian, Proposition \ref{prop:Sol Lyapunov eq} specifically provides a closed form expression for the observability 
Gramian in terms of the eigenvalues and eigenvectors of $\bm A$.
In this approach we use (\ref{eq:Solution of lyapunov equation}) to compute $\mathcal{P}={\bm B}^\top {\bm X} {\bm B}$
and discuss in what circumstances the limit $\epsilon\rightarrow0$ can be taken safely to recover the physically relevant quantities. \\

\begin{assumption}[Uniform damping to inertia ratio]\label{ass:Homogeneous inertia damping ratio}\ \\
 All synchronous machines have uniform damping over inertia ratios $d_i/m_i= \gamma>0~\forall i$.
 This assumption makes the computation of quadratic performance measures analytically tractable, and is a standard one
 in the literature~\cite{bamieh2013price,Gayme15,Gayme16,Poola17,paganini2017global,coletta2017performance}. 
 Machine measurements indicate that the ratio $d_i/m_i$ varies by at most an order of magnitude from rotating machine to rotating 
 machine \cite{KouMachineReport}.\newline
\end{assumption}

\begin{proposition}[Diagonalization of $\bm A$]\label{prop:Trasformation}\ \\
 Under the assumption of uniform damping to inertia ratios, the left and right transformation matrices 
 $\bm T_L$ and $\bm T_R$ diagonalizing ${\bm A}$ can be expressed in terms of the eigenvectors of 
 ${\bm L}_\textrm{M}$ through the linear transformations given below in (\ref{eq:T_R, T_L}), (\ref{eq:S_R}), and (\ref{eq:S_L}) . 
\end{proposition}
\begin{proof}
For uniform damping to inertia ratios one has that ${\bm M}^{\text{-}1}\bm D=\gamma\mathbb{I}$. 
Thus ${\bm M}^{\text{-}1}\bm D$ and ${\bm L}_\textrm{M}$ commute and share a common eigenbasis.
Since ${\bm L}_\textrm{M}$ is symmetric, it has a real spectrum with eigenvalues
denoted by $\lambda^\textrm{M}_i$, and it is diagonalized by an orthogonal matrix ${\bm T}_\textrm{M}$ 
\begin{equation}
{\bm T}_\textrm{M}^\top{\bm L}_\textrm{M}{\bm T}_\textrm{M} = 
{\bm \Lambda}_\textrm{M} :=\textrm{diag}(\{\lambda_i^\textrm{M}\})\,.
\end{equation}
From the similarity transformation 
\begin{equation}\label{eq:Similarity trasformation}
\setlength\arraycolsep{2pt}
\begin{bmatrix}
        {\bm T}_\textrm{M}^\top & 0 & 0 \\
        0 & {\bm T}_\textrm{M}^\top & 0 \\
        0 & 0 & 1
\end{bmatrix} 
{\bm A}
\setlength\arraycolsep{2pt}
\begin{bmatrix}
        {\bm T}_\textrm{M} & 0 & 0\\
        0 & {\bm T}_\textrm{M} & 0\\
        0 & 0 & 1
    \end{bmatrix}
    =
\begin{bmatrix}
        0 & \mathbb{I} & 0\\
        -{\bm \Lambda}_\textrm{M} & -\gamma\mathbb{I} & {\bm T}_\textrm{M}^\top{\bm M}^{\text{-}1/2}\bm p \\
        0 & 0 & -\tau^{\text{-}1}
\end{bmatrix},
\end{equation}
one easily obtains the eigenvalues of $\bm A$ which are
\begin{equation}\label{eq:Eigenvalues A}
\left\{\mu_1^+,\ldots\mu_N^+,\mu_1^-,\ldots\mu_N^-,-\tau^{\text{-1}}\right\}
\end{equation}
with 
\begin{equation}\label{eq:def mu_i}
 \mu_i^{\pm}=\frac{1}{2}\left(-\gamma\pm\Gamma_i\right)\,, \quad \Gamma_i=\sqrt{\gamma^2-4\lambda_i^\textrm{M}}\,.
\end{equation}
From the last row of the right-hand side of (\ref{eq:Similarity trasformation}) one straightforwardly concludes that $-\tau^{\text{-}1}$
is an eigenvalue. The remaining $\mu_i^\pm$'s eigenvalues actually are the eigenvalues of the upper left $2N\times2N$ block of the right-hand side of 
(\ref{eq:Similarity trasformation}). This can be easily seen after the appropriate index reordering of this block and using that $\mathbb{I}$, 
$-{\bm\Lambda}_\textrm{M}$, and $-\gamma\mathbb{I}$ are all diagonal.

For $\Gamma_i\neq0$, the full transformation which diagonalizes $\bm A$
\begin{equation}
\setlength\arraycolsep{3pt}
{\bm T}_L {\bm A}{\bm T}_R = \begin{bmatrix}
        \textrm{diag}(\{ \mu_i^+\}) & 0 & 0\\
        0 & \textrm{diag}(\{ \mu_i^-\}) & 0\\
        0 & 0 & -\tau^{\text{-}1}\\
    \end{bmatrix}\,,
\end{equation}
and which fulfills the bi-orthogonality condition ${\bm T}_L {\bm T}_R={\bm T}_R {\bm T}_L=\mathbb{I}$ is
given by
\begin{equation}\label{eq:T_R, T_L}
 {\bm T}_R = 
 \setlength\arraycolsep{3pt}
 \begin{bmatrix}
        {\bm T}_\textrm{M} & 0 & 0 \\
        0 & {\bm T}_\textrm{M} & 0 \\
        0 & 0 & 1
 \end{bmatrix} {\bm S}_R \,,\qquad
 {\bm T}_L = 
 {\bm S}_L 
 \setlength\arraycolsep{3pt}
 \begin{bmatrix}
        {\bm T}^\top_\textrm{M} & 0 & 0\\
        0 & {\bm T}^\top_\textrm{M} & 0\\
        0 & 0 & 1\\
 \end{bmatrix}\,,\qquad
\end{equation}
with 
\begin{equation}\label{eq:S_R}
 {\bm S}_R = 
 \medmath{
 \setlength\arraycolsep{2pt}
    \begin{bmatrix}
        \textrm{diag}(\{\frac{1}{\sqrt{\Gamma_j}}\}) & \textrm{diag}(\{\frac{\text{i}}{\sqrt{\Gamma_j}}\}) & \textrm{vec}(\{\frac{\tau^2[{\bm T}_\textrm{M}^\top{\bm M}^{\text{-}1/2}\bm p]_j}{1-\gamma\tau+\tau^2\lambda_j^M}\})\\
        \textrm{diag}(\{\frac{\mu_j^+}{\sqrt{\Gamma_j}}\}) & \textrm{diag}(\{\frac{\text{i}\mu_j^-}{\sqrt{\Gamma_j}}\} ) & \textrm{vec}(\{\frac{-\tau [{\bm T}_\textrm{M}^\top{\bm M}^{\text{-}1/2}\bm p]_j}{1-\gamma\tau+\tau^2\lambda_j^M}\})\\
        0 & 0 & 1
    \end{bmatrix}\,,\qquad}
\end{equation}
and
\begin{equation}\label{eq:S_L}
 {\bm S}_L = 
 \medmath{
 \setlength\arraycolsep{1pt}
    \begin{bmatrix}
        \textrm{diag}(\{\frac{-\mu_j^-}{\sqrt{\Gamma_j}}\}) & \textrm{diag}(\{\frac{1}{\sqrt{\Gamma_j}}\}) & \textrm{vec}(\{\frac{\tau(1+\tau\mu_j^-)[{\bm T}_\textrm{M}^\top{\bm M}^{\text{-}1/2}\bm p]_j}{\sqrt{\Gamma_j}(1-\gamma\tau+\tau^2\lambda_j^M)}\})\\
        \textrm{diag}(\{\frac{-\text{i}\mu_j^+}{\sqrt{\Gamma_j}}\}) & \textrm{diag}(\{\frac{\text{i}}{\sqrt{\Gamma_j}}\} ) & \textrm{vec}(\{\frac{\text{i}\tau(1+\tau\mu_j^+)[{\bm T}_\textrm{M}^\top{\bm M}^{\text{-}1/2}\bm p]_j}{\sqrt{\Gamma_j}(1-\gamma\tau+\tau^2\lambda_j^M)}\})\\
        0 & 0 & 1 \\
    \end{bmatrix}}\,.
\end{equation}
\end{proof}

Equations~(\ref{eq:T_R, T_L}), (\ref{eq:S_R}) and (\ref{eq:S_L}) relate the eigenvectors of $\bm A$ to those of ${\bm L}_\textrm{M}$.
Combining this result with the result of Proposition \ref{prop:Sol Lyapunov eq} we express the observability Gramian of (\ref{eq:Solution of lyapunov equation})
in terms of the eigenvectors of ${\bm L}_\textrm{M}$. \\

\begin{proposition}[Generic performance measure]\label{prop: Perf measure}\ \\
 Consider the power system model defined in (\ref{eq:Swing3}) and satisfying Proposition \ref{prop:Marginal mode assumption}.
 Under the assumption of uniform damping to inertia ratios $d_i/m_i=\gamma~\forall i$, the quadratic performance measure $\mathcal{P}$
 defined in~(\ref{eq:objective function}) is given by
\begin{equation}\label{eq:General perfromance}
\begin{array}{l}
\displaystyle \mathcal{P}=\eta_0^2\sum_{l,q=1}^N[{\bm T}_\textrm{M}^\top{\bm M}^{\text{-}1/2}\bm p]_l[{\bm T}_\textrm{M}^\top{\bm M}^{\text{-}1/2}\bm p]_q \\
\quad\quad\times \left\{({\bm T}_\textnormal{M}^\top {\bm Q}^{(1,1)}_\textrm{M}{\bm T}_\textnormal{M})_{lq}~f
 +({\bm T}_\textnormal{M}^\top {\bm Q}^{(2,2)}_\textrm{M}{\bm T}_\textnormal{M})_{lq}~g \right\}\,,
\end{array}
\end{equation}
where $f \equiv f(\tau,\gamma,\lambda_l^\textrm{M},\lambda_q^\textrm{M})$ and $g \equiv g(\tau,\gamma,\lambda_l^\textrm{M},\lambda_q^\textrm{M})$ 
are scalar functions of $\tau,\gamma,\lambda_l^\textrm{M}$, and $\lambda_q^\textrm{M}$, given in Appendix \ref{app:f and g},
and where $\lambda_l^\textnormal{M}$ and $\bm T_\textnormal{M}$ are the eigenvalues and the orthogonal matrix diagonalizing 
 ${\bm L}_\textrm{M}$.
\end{proposition}

The proof of Proposition \ref{prop: Perf measure} will be given elsewhere~\cite{Col19}.

The assumption that Proposition \ref{prop:Marginal mode assumption} holds implies that the eigenvalues $\lambda_l^\textrm{M}$ in Proposition \ref{prop: Perf measure}
are functions of the regularizing parameter $\epsilon$, that is $\lambda_l^\textrm{M}\equiv\lambda_l^\textrm{M}(\epsilon)$.
While (\ref{eq:General perfromance}) formally holds for $\epsilon\neq0$, we illustrate below how for performance measures $\bm Q$ such that 
$[\bm u^{(1)},0,0]^\top\in\textrm{ker}({\bm Q})$ one can safely take the limit $\epsilon\rightarrow0$.

\section{Phase coherence}\label{sec:Applications}

Given the average phase deviation $\tilde{\varphi}=\sum_{i=1}^N\varphi_i/N$, the phase coherence metric
$\mathcal{P}_{\varphi}=\int_0^\infty \sum_{i=1}^N({\varphi}_i(t)-\tilde{\varphi}(t))^2 \dd t$ measures the transient voltage phase variance.
It is obtained taking ${\bm Q}^{(1,1)}=\mathbb{I}-\bm{u}^{(1)}{\bm{u}^{(1)}}^\top$ and ${\bm Q}^{(2,2)}=0$ in (\ref{eq:objective function}). 
For this performance measure $[\bm u^{(1)},0,0]^\top\in\textrm{ker}({\bm Q})$ and we show how one can safely let the regularizing parameter $\epsilon$
go to zero.
In the case of uniform inertia, $m_i=m~\forall i$,
we have ${\bm Q}^{(1,1)}_\textrm{M}=(\mathbb{I}-\bm{u}^{(1)}{\bm{u}^{(1)}}^\top)/m$ and ${\bm Q}^{(2,2)}_\textrm{M}=0$ which, once inserted 
in (\ref{eq:General perfromance}) gives
\begin{align}\label{eq:phasecoherence}
\displaystyle \mathcal{P}_{\varphi}=\frac{\eta_0^2}{m^2}\sum_{l,q=1}^N & [{\bm T}_\textrm{M}^\top\bm p]_l[{\bm T}_\textrm{M}^\top\bm p]_q 
({\bm T}_\textnormal{M}^\top(\mathbb{I}-\bm{u}^{(1)}{\bm{u}^{(1)}}^\top){\bm T}_\textnormal{M})_{lq} \nonumber  \\
& \quad \times f(\tau,\gamma,\lambda_l^\textrm{M},\lambda_q^\textrm{M})\,.
\end{align}
For homogeneous inertia values, we also have that ${\bm L}_M={\bm L}/m$. It follows that both matrices have same eigenvectors ${\bm T}_\textrm{M}\equiv{\bm T}$,
while their eigenvalues differ by a factor $m$, $\lambda_l^\textrm{M}=\lambda_l/m$.
Using the orthogonality conditions $({\bm T}^\top {\bm T})_{lq}=\delta_{lq}$
and $({\bm T}^\top \bm{u}^{(1)}{\bm{u}^{(1)}}^\top {\bm T})_{lq}=\delta_{lq}\delta_{l1}$,
(\ref{eq:phasecoherence}) simplifies to
\begin{equation}
\displaystyle \mathcal{P}_{\varphi}=\frac{\eta_0^2}{m^2}\sum_{l\geq2}^N{[{\bm T}^\top\bm p]}^2_l 
~f(\tau,\gamma,\lambda_l/m,\lambda_l/m)\,.
\end{equation}
When $l=q$, the function $f(\tau,\gamma,\lambda_l^\textrm{M},\lambda_q^\textrm{M})$ simplifies to
$f(\tau,\gamma,\lambda_l^\textrm{M},\lambda_l^\textrm{M})=\tau(1+\gamma\tau)/[2\lambda_l^\textrm{M}\gamma(\tau^{\text{-}1}+\gamma+\lambda_l^\textrm{M}\tau)]$
and the performance measure finally becomes
\begin{equation}
\mathcal{P}_{\varphi}=\sum_{l\geq2}^N{[{\bm T}^\top\bm p]}^2_l \frac{(m+d\tau)}{\lambda_ld(\tau^{\text{-}1}m+d+\lambda_l\tau)}\,.
\end{equation}
We note that since the summation index $l\geq2$, this expression is well behaved also if the regularizing parameter $\epsilon$ is set to zero.
In the specific case where the power injection fluctuation is localized at a single node labeled $\alpha$, i.e. $\bm p = p\hat{\bm e}_\alpha$,
we have
\begin{equation}\label{eq:perf phase single node uniform m}
\mathcal{P}_{\varphi}=\sum_{l\geq2}^Np^2{u_\alpha^{(l)}}^2 \frac{(m+d\tau)}{\lambda_ld(\tau^{\text{-}1}m+d+\lambda_l\tau)}\,,
\end{equation}
where ${\bm u}^{(l)}$ is the eigenvector of the network's Laplacian $\bm L$, associated to the eigenvalue $\lambda_l$.

We next interpret our result (\ref{eq:perf phase single node uniform m}) from a graph-theoretic perspective.
We show that depending on the correlation time scale $\tau$ and  on the measure considered,
the transient performance is either independent of the location of the noisy node or is determined
by the resistance closeness centrality of the noisy node.

The effective resistance distance between any two nodes $i$ and $j$ of the network is defined as 
$\Omega_{ij}={L}^\dagger_{ii}+{L}^\dagger_{jj}-2{L}^\dagger_{ij}$, where ${\bm L}^\dagger$ 
is the Moore-Penrose pseudoinverse of the network's Laplacian matrix ${\bm L}$ \cite{Klein93, Stephenson89}. 
It is known as the \textit{resistance} distance because if one replaces the network edges by resistors
with a resistance $R_{ij}=1/b_{ij}$, then $\Omega_{ij}$ is equal to the
equivalent network resistance when a current is injected at node $i$ and extracted at node $j$ with no injection anywhere
else. 
The pseudoinverse of ${\bm L}$ is given by ${\bm L}^\dagger = {\bm T}\textrm{diag}(\{0,\lambda_2^{\text{-}1},\ldots,\lambda_N^{\text{-}1}\}){\bm T}^\top$.
This allows to rewrite the resistance distance in terms of the eigenvalues and eigenvectors of $\bm L$ \cite{klein1997graph, xiao2003resistance}
\begin{equation}\label{eq: resistance distance}
 \Omega_{ij}=\sum_{l\geq2}^N\lambda_l^{\text{-}1}(u_{i}^{(l)}-u_{j}^{(l)})^2\,.
\end{equation}
The resistance distance closeness centrality of node $\alpha$, $C_\alpha$, is the inverse average distance separating node $\alpha$ from the rest of the network
$C_\alpha^{\text{-}1}=\sum_{j=1}^N{\Omega_{\alpha j}}/N$~\cite{Bozzo13}.
Using (\ref{eq: resistance distance}), the inverse closeness centrality is given by
\begin{equation}\label{eq:Closeness centrality}
  C_\alpha^{\text{-}1}=\sum_{l\geq2}^N\frac{{u_\alpha^{(l)}}^2}{\lambda_l}+\frac{1}{N}\sum_{l\geq2}^N\frac{1}{\lambda_l}\,,
\end{equation}
where we have used that $\sum_{j=1}^N {u_j^{(l)}}^2={\bm{u}^{(l)}}^\top \bm{u}^{(l)}=1$, and that $\sum_{j=1}^N u_j^{(l)}=0$ for $l \ne 1$, since ${\bm u}^{(l)}\perp{\bm u}^{(1)}$.
We note that only the first term in the right-hand side of (\ref{eq:Closeness centrality}) depends on $\alpha$.
The second term is proportional to the networks Kirchoff's index $\mathit{Kf}_1=N\sum_{l\geq2}^N\lambda_l^{\text{-}1}$, and is thus 
independent of the location of the noisy node.

We next consider the performance measures $\mathcal{P}_\varphi$ in the limit when
the correlation time $\tau$ is much shorter than any 
characteristic time scale of the swing equation.
Expanding (\ref{eq:perf phase single node uniform m})  
in the limit $\tau\ll1$ we get
\begin{align}\label{eq:small tau}
\displaystyle\mathcal{P}_{\varphi}&\approx \frac{\tau p^2}{d}\sum_{l\geq2}^N \frac{{u_\alpha^{(l)}}^2}{\lambda_l} +\mathcal{O}(\tau^2) \nonumber \\
 &= \frac{\tau p^2}{d}\left[C_\alpha^{\text{-}1}-\mathit{Kf}_1/N^2\right] +\mathcal{O}(\tau^2)\,.
\end{align}
We see that for colored noise injection at node $\alpha$ 
with fast decaying correlations, the phase coherence is proportional to the inverse closeness centrality of 
the noisy node. 

In the opposite limit $\tau \gg 1$ a Taylor expansion of Equation
(\ref{eq:perf phase single node uniform m}) gives 
\begin{align}\label{eq:large tau}
\displaystyle\mathcal{P}_{\varphi} &\approx p^2 \sum_{l\geq2}^N\frac{{u_\alpha^{(l)}}^2}{\lambda_l^2} +\mathcal{O}(\tau^{\text{-}1})\,.
\end{align}
The phase coherence measure $\mathcal{P}_\varphi$ still depends on the location of the noisy node $\alpha$, but this time with a more involved expression of network related quantities. 

\section{CONCLUSION}\label{sec:Conclusion}

Our results illustrate how finite-time 
correlations in power fluctuations affect the transient performance. We have shown how, depending on the correlation time scale
$\tau$, 
performance measures can change qualitatively form being network independent to network dependent.
Our analytical results clearly emphasize that the resistance distance, and the associated resistance closeness centrality are the 
physically relevant measures of node criticality.

Compared to white noise, colored noise inputs provide a better description of the stochastic fluctuations of 
renewable generation. Future works should  try to improve this modeling assumption and incorporate 
the non Gaussian character of renewable generation~\cite{schmietendorf2017impact}.

\section*{ACKNOWLEDGMENTS}
 We thank F.~D\"{o}rfler for useful discussions.
 This work was supported by the Swiss National Science Foundation
 under an AP Energy Grant.

\section{APPENDIX}
\subsection{Colored Noise from Gaussian white noise}\label{app:Colored noise}
In this section we illustrate how the augmented dynamical model (\ref{eq:Swing2}), provides the framework to treat exponentially decorrelating noise.
Consider the differential equation
\begin{equation}\label{eq:Generating colored noise}
 \dot{\eta}(t)=-\tau^{-1}\eta(t) + \eta_0\xi(t)\,,
\end{equation}
where $\eta_0=\sqrt{2/\tau}$ and $\xi(t)$ is a Gaussian white noise signal, such that $\mathbb{E}[\xi(t_1)\xi(t_2)]=\delta(t_1-t_2)$.
Solving (\ref{eq:Generating colored noise}), with initial condition $\eta(0)=0$ leads to
\begin{equation}
 {\eta}(t)=\eta_0\int_0^t\xi(s)e^{(s-t)/\tau}\dd s\,,
\end{equation}
from which one obtains
\begin{equation}
 \mathbb{E}[\eta(t_1)\eta(t_2)]=\eta_0^2 \frac{\tau}{2}\left[e^{-|t_1-t_2|/\tau}-e^{-(t_1+t_2)/\tau}\right]\,,
\end{equation}
which simplifies to 
\begin{equation}
 \mathbb{E}[\eta(t_1)\eta(t_2)]=e^{-|t_1-t_2|/\tau}\,,
\end{equation}
for $t_1, t_2 \gg \tau$.
\subsection{Generic performance measure coefficients}\label{app:f and g}
The functions $f$
and $g$  in (\ref{eq:General perfromance}) are given by
\begin{align}\label{eq:full f}
 f(\tau,\gamma,\lambda_l^\textrm{M},\lambda_q^\textrm{M}) &=\medmath{(1+\gamma\tau+\lambda_l^\textrm{M}\tau^2)^{-1}(1+\gamma\tau+\lambda_q^\textrm{M}\tau^2)^{-1}\frac{\tau^2}{2}} \nonumber \\[2mm]
\times&\quad\medmath{\left[ 
\frac{8\gamma^2\tau+4\gamma+2\gamma\tau^2\left(2\gamma^2+\lambda_l^\textrm{M}+\lambda_q^\textrm{M}\right)}{2\gamma^2\left(\lambda_l^\textrm{M}+\lambda_q^\textrm{M}\right)+\left(\lambda_l^\textrm{M}-\lambda_q^\textrm{M}\right)^2}+\tau^3
\right]}\,,
\end{align}

and 
\begin{equation}\label{eq:full g}
\begin{array}{l}
 g(\tau,\gamma,\lambda_l^\textrm{M},\lambda_q^\textrm{M}) = \medmath{(1+\gamma\tau+\lambda_l^\textrm{M}\tau^2)^{-1}(1+\gamma\tau+\lambda_q^\textrm{M}\tau^2)^{-1}\frac{\tau^2}{2}}\\[2mm]
\times \,\medmath{\frac{\left[2\gamma^2\tau\left(\lambda_l^\textrm{M}+\lambda_q^\textrm{M}\right)-\tau\left(\lambda_l^\textrm{M}-\lambda_q^\textrm{M}\right)^2+2\gamma\left(\lambda_l^\textrm{M}+\lambda_q^\textrm{M}+2\tau^2\lambda_l^\textrm{M}\lambda_q^\textrm{M}\right)\right]}
{2\gamma^2\left(\lambda_l^\textrm{M}+\lambda_q^\textrm{M}\right)+\left(\lambda_l^\textrm{M}-\lambda_q^\textrm{M}\right)^2}}\,.
\end{array}
\end{equation}



\begin{thebibliography}{10}
\providecommand{\url}[1]{#1}
\csname url@samestyle\endcsname
\providecommand{\newblock}{\relax}
\providecommand{\bibinfo}[2]{#2}
\providecommand{\BIBentrySTDinterwordspacing}{\spaceskip=0pt\relax}
\providecommand{\BIBentryALTinterwordstretchfactor}{4}
\providecommand{\BIBentryALTinterwordspacing}{\spaceskip=\fontdimen2\font plus
\BIBentryALTinterwordstretchfactor\fontdimen3\font minus
  \fontdimen4\font\relax}
\providecommand{\BIBforeignlanguage}[2]{{%
\expandafter\ifx\csname l@#1\endcsname\relax
\typeout{** WARNING: IEEEtran.bst: No hyphenation pattern has been}%
\typeout{** loaded for the language `#1'. Using the pattern for}%
\typeout{** the default language instead.}%
\else
\language=\csname l@#1\endcsname
\fi
#2}}
\providecommand{\BIBdecl}{\relax}
\BIBdecl

\bibitem{Bialek08}
J.~Machowski, J.~W. Bialek, and J.~R. Bumby, \emph{Power system dynamics:
  stability and control}, John Wiley, Chichester, U.K, 2008.

\bibitem{AnnualEnOutlook}
``Power systems of the future: The case for energy storage, distributed
  generation, and microgrids,'' {\relax IEEE Smart Grid, Tech. Rep. Nov. 2012}.

\bibitem{Bamieh12}
B.~Bamieh, M.~R. Jovanovi{\'c}, P.~Mitra, and S.~Patterson, ``Coherence in
  large-scale networks: Dimension-dependent limitations of local feedback,''
  \emph{IEEE Transactions on Automatic Control}, vol.~57, no.~9, pp.
  2235--2249, 2012.

\bibitem{Summers15}
T.~Summers, I.~Shames, J.~Lygeros, and F.~D{\"o}rfler, ``Topology design for
  optimal network coherence,'' in \emph{European Control Conference}.\hskip 1em
  plus 0.5em minus 0.4em\relax IEEE, pp. 575--580, 2015.

\bibitem{Siami14}
M.~Siami and N.~Motee, ``Systemic measures for performance and robustness of
  large-scale interconnected dynamical networks,'' in \emph{53rd Annual
  Conference on Decision and Control}.\hskip 1em plus 0.5em minus 0.4em\relax
  IEEE, pp. 5119--5124, 2014.

\bibitem{Fardad14}
M.~Fardad, F.~Lin, and M.~R. Jovanovi{\'c}, ``Design of optimal sparse
  interconnection graphs for synchronization of oscillator networks,''
  \emph{IEEE Transactions on Automatic Control}, vol.~59, no.~9, pp.
  2457--2462, 2014.

\bibitem{bamieh2013price}
B.~Bamieh and D.~F. Gayme, ``The price of synchrony: Resistive losses due to
  phase synchronization in power networks,'' in \emph{American Control
  Conference}.\hskip 1em plus 0.5em minus 0.4em\relax IEEE, pp. 5815--5820, 2013.

\bibitem{Gayme15}
E.~Tegling, B.~Bamieh, and D.~F. Gayme, ``The price of synchrony: Evaluating
  the resistive losses in synchronizing power networks,'' \emph{IEEE
  Transactions on Control of Network Systems}, vol.~2, no.~3, pp. 254--266,
  2015.

\bibitem{Gayme16}
T.~W. Grunberg and D.~F. Gayme, ``Performance measures for linear oscillator
  networks over arbitrary graphs,'' \emph{IEEE Transactions on Control of
  Network Systems}, vol.~PP, no.~99, pp. 1--1, 2016.

\bibitem{Poola17}
B.~K. Poolla, S.~Bolognani, and F.~D\"{o}rfler, ``Optimal placement of virtual
  inertia in power grids,'' \emph{IEEE Transactions on Automatic Control},
  vol.~PP, no.~99, pp. 1--1, 2017.

\bibitem{paganini2017global}
F.~Paganini and E.~Mallada, ``Global performance metrics for synchronization of
  heterogeneously rated power systems: The role of machine models and
  inertia,'' in \emph{55th Annual Allerton Conference on Communication,
  Control, and Computing}, pp. 324--331, 2017.

\bibitem{coletta2017performance}
T.~Coletta and P.~Jacquod, ``Performance measures in electric power networks
  under line contingencies,'' \emph{arXiv preprint arXiv:1711.10348}, 2017.

\bibitem{Zhou96}
K.~Zhou, J.~Doyle, and K.~Glover, ``Robust and optimal control,'' vol. 40, Prentice Hall, Upper Saddle River, NJ, 1996.

\bibitem{schmietendorf2017impact}
K.~Schmietendorf, J.~Peinke, and O.~Kamps, ``The impact of turbulent renewable
  energy production on power grid stability and quality,'' \emph{The European
  Physical Journal B}, vol.~90, no.~11, p. 222, 2017.


\bibitem{zare2017colour}
A.~Zare, M.~R. Jovanovi{\'c}, and T.~T. Georgiou, ``Colour of turbulence,''
  \emph{Journal of Fluid Mechanics}, vol. 812, pp. 636--680, 2017.

\bibitem{Fox88} R.F. Fox, I.R. Gatland, R. Roy, and G. Vemuri, 
"Fast, accurate algorithm for numerical simulation of exponentially correlated colored noise", Phys. Rev. A {\bf 38}, 5938 (1988).


\bibitem{Klein93}
D.~J. Klein and M.~Randi{\'c}, ``Resistance distance,'' \emph{Journal of
  Mathematical Chemistry}, vol.~12, no.~1, pp. 81--95, 1993.
  
\bibitem{KouMachineReport}
\BIBentryALTinterwordspacing
P.~M. G.~Kou, S. W.~Hadley and Y.~Liu, ``Developing generic dynamic models for
  the 2030 eastern interconnection grid,'' Oak Ridge National Laboratory, Tech.
  Rep., Dec. 2013. [Online]. Available: \url{http://www.osti.gov/scitech/}
\BIBentrySTDinterwordspacing

\bibitem{Col19} T. Coletta, B. Bamieh and Ph. Jacquod, in preparation.

\bibitem{Stephenson89}
K.~Stephenson and M.~Zelen, ``Rethinking centrality: Methods and examples,''
  \emph{Social Networks}, vol.~11, no.~1, pp. 1 -- 37, 1989.

\bibitem{klein1997graph}
D.~J. Klein, ``Graph geometry, graph metrics and {W}iener,'' \emph{Commun.
  Math. Comput. Chem.}, no.~35, pp. 7--27, 1997.

\bibitem{xiao2003resistance}
W.~Xiao and I.~Gutman, ``Resistance distance and laplacian spectrum,''
  \emph{Theoretical Chemistry Accounts}, vol. 110, no.~4, pp. 284--289, 2003.

\bibitem{Bozzo13}
E.~Bozzo and M.~Franceschet, ``Resistance distance, closeness, and
  betweenness,'' \emph{Social Networks}, vol.~35, no.~3, pp. 460 -- 469, 2013.

\end{thebibliography}
\end{document}